\documentclass{amsart}

\usepackage[T1]{fontenc}
\usepackage{amsmath}
\usepackage{amssymb}
\usepackage{microtype}
\usepackage[hidelinks]{hyperref}

\numberwithin{equation}{section}

\newtheorem{theorem}{Theorem}[section]
\newtheorem{proposition}[theorem]{Proposition}
\newtheorem{lemma}[theorem]{Lemma}

\theoremstyle{remark}

\newcommand{\NN}{\mathbb{N}}
\newcommand{\Apery}{\operatorname{Ap}}

\title[Power-sum obstruction to cyclotomicity]
{A Power-Sum Obstruction to Cyclotomicity in a Family of Symmetric Numerical Semigroups}

\author[Zhi-Lin Zhang]{Zhi-Lin Zhang}
\address{Independent Researcher, Taipei, Taiwan}
\email{hsa00000@gmail.com}



\date{}

\begin{document}

\begin{abstract}
For positive integers \(q\) and \(m\) with \(m\geq2q+3\), consider the
symmetric numerical semigroup
\[
S_{m,q}
=
\left\langle
m,\,m+1,\,qm+2q+2,\,qm+2q+3,\,\ldots,\,qm+m-1
\right\rangle.
\]
Ciolan, Garc\'ia-S\'anchez, and Moree asked whether every member of this
family with embedding dimension at least \(4\) is noncyclotomic.  We
answer this question affirmatively for every \(q\geq1\), including
an independent proof of the previously known case \(q=1\).

Let \(P_{m,q}\) be the semigroup polynomial, and set
\[
t=m-2q-3,
\qquad
L=2(q+1)(m+1)-1,
\qquad
D=\deg P_{m,q}=2qm+2q+2.
\]
For \(m\geq2q+4\), let \(\rho_1,\ldots,\rho_D\) be the roots of
\(P_{m,q}\), counted with multiplicity.  An explicit computation of a
single coefficient of the formal logarithm gives
\[
\left|\sum_{j=1}^{D}\rho_j^{-L}\right|
=
L\left|[x^L]\log P_{m,q}(x)\right|
\geq Lt-1
>D.
\]
The strict inequality forces \(P_{m,q}\) to have a root off the unit
circle.  Hence \(S_{m,q}\) is noncyclotomic whenever \(m\geq2q+4\).
The boundary case
\(m=2q+3\) has embedding dimension \(3\) and is cyclotomic.  Therefore
\[
S_{m,q}\text{ is cyclotomic}
\quad\Longleftrightarrow\quad
m=2q+3.
\]
\end{abstract}

\maketitle

\section{Introduction}

Let \(\NN=\{0,1,2,\ldots\}\).  A \emph{numerical semigroup} is a
cofinite additive submonoid of \((\NN,+)\).  Its unique minimal generating
set has cardinality \(e(S)\), the \emph{embedding dimension} of \(S\).
For \(S\neq\NN\), the largest integer outside \(S\) is its
\emph{Frobenius number} \(F(S)\).  The semigroup is \emph{symmetric} if
\[
n\in S
\quad\Longleftrightarrow\quad
F(S)-n\notin S
\qquad(n\in\mathbb Z).
\]

The Hilbert series and semigroup polynomial of \(S\) are
\[
H_S(x)=\sum_{s\in S}x^s,
\qquad
P_S(x)=(1-x)H_S(x).
\]
The series \(H_S(x)\) is rational, and \(P_S(x)\) is a monic polynomial
in \(\mathbb Z[x]\) with constant term \(1\).  Following Ciolan,
Garc\'ia-S\'anchez, and Moree \cite{CiolanGarciaSanchezMoree}, a
numerical semigroup \(S\) is called \emph{cyclotomic} if \(P_S(x)\) is a
finite product of cyclotomic polynomials.  In particular, every zero of
\(P_S\) then lies on the unit circle.

For positive integers \(q\) and \(m\) with \(m\geq2q+3\), define
\begin{equation}
\label{eq:family}
S_{m,q}
=
\left\langle
m,\,m+1,\,qm+2q+2,\,qm+2q+3,\,\ldots,\,qm+m-1
\right\rangle,
\end{equation}
and write \(P_{m,q}(x)=P_{S_{m,q}}(x)\).  This family arises from
Rosales's construction of symmetric numerical semigroups with prescribed
multiplicity and embedding dimension \cite{Rosales}.  In the present
notation, its symmetry and embedding-dimension formula follow from
\cite[Lemma~4.22]{RosalesGarciaSanchez}; in particular,
\begin{equation}
\label{eq:embedding-dimension}
e(S_{m,q})=m-2q.
\end{equation}
Ciolan, Garc\'ia-S\'anchez, and Moree recorded these facts together with
the Ap\'ery-set description used below; see
\cite[Example~3]{CiolanGarciaSanchezMoree}.  They asked whether every
member of this family with embedding dimension at least \(4\), or
equivalently \(m\geq2q+4\), is noncyclotomic
\cite[Problem~2]{CiolanGarciaSanchezMoree}.

Sawhney and Stoner settled the case \(q=1\) by combining an asymptotic
result with a finite verification
\cite[Theorem~1 and Appendix~A]{SawhneyStoner}.  Related criteria based
on logarithmic derivatives evaluated at \(x=\pm1\) were developed by
Herrera-Poyatos and Moree, with applications to symmetric noncyclotomic
numerical semigroups \cite{HerreraPoyatosMoree}.  Further work on cyclotomic numerical semigroups appears in
\cite{BorziHerreraPoyatosMoree,
CiolanGarciaSanchezHerreraPoyatosMoree}.  For a recent survey of open
problems in numerical semigroup theory, see
\cite{MoscarielloSammartano}.

Our argument instead uses a single coefficient in the formal
expansion of \(\log P_{m,q}(x)\) at \(x=0\).  If
\(\rho_1,\ldots,\rho_D\) are the roots of \(P_{m,q}\), counted with
multiplicity, where \(D=\deg P_{m,q}\), then
\[
-r[x^r]\log P_{m,q}(x)
=
\sum_{j=1}^{D}\rho_j^{-r}.
\]
At the index
\[
r=L=2(q+1)(m+1)-1,
\]
we show that the absolute value of this power sum is greater than \(D\).
This contradicts the triangle-inequality bound that would hold if every
root lay on the unit circle.  The resulting coefficient argument applies
to every \(q\geq1\), and in particular gives an independent proof of the
case \(q=1\).  To the best of our knowledge, no previous result covers
the full range \(q\geq2\) in the family \eqref{eq:family}.

\begin{theorem}
\label{thm:classification}
Let \(q\) and \(m\) be positive integers with \(m\geq2q+3\).  Then
\[
S_{m,q}\text{ is cyclotomic}
\quad\Longleftrightarrow\quad
m=2q+3.
\]
In particular, Problem~2 of \cite{CiolanGarciaSanchezMoree} has an
affirmative answer.
\end{theorem}

Section~2 derives the required Ap\'ery-set formulas and records the
power-sum identity for reciprocal roots used to detect a zero off the unit circle.
Section~3 computes the coefficient at \(L\) and proves
Theorem~\ref{thm:classification}.

\section{Ap\'ery-set formulas and power sums of reciprocal roots}

Let \(S\) be a numerical semigroup and let \(a\in S\setminus\{0\}\).
The Ap\'ery set of \(S\) with respect to \(a\) is
\[
\Apery(S;a)=\{s\in S:s-a\notin S\}.
\]
It contains one element in each residue class modulo \(a\), and hence
\[
H_S(x)
=
\frac{W_{S,a}(x)}{1-x^a},
\qquad
P_S(x)
=
\frac{(1-x)W_{S,a}(x)}{1-x^a},
\qquad
W_{S,a}(x)
:=
\sum_{w\in\Apery(S;a)}x^w.
\]
If \(S\neq\NN\), then
\begin{equation}
\label{eq:degree-from-apery}
\deg P_S
=
F(S)+1
=
\max\Apery(S;a)-a+1.
\end{equation}
These standard Ap\'ery-set identities and the degree formula are recalled
in \cite[Section~2]{CiolanGarciaSanchezMoree}.

For the family \(S_{m,q}\), the Ap\'ery-set computation in
\cite[Example~3]{CiolanGarciaSanchezMoree} gives the formulas needed
for the coefficient calculation in Section~3.

\begin{proposition}
\label{prop:apery-family}
Let \(q\) and \(m\) be positive integers with \(m\geq2q+3\), and set
\[
t=m-2q-3,
\qquad
v=qm+2q+2,
\qquad
A_t(x)=1+x+\cdots+x^t.
\]
Then
\begin{equation}
\label{eq:apery-set-family}
\Apery(S_{m,q};m)
=
\{k(m+1):0\leq k\leq2q+1\}
\mathbin{\cup}
\{qm+r:2q+2\leq r\leq m-1\}.
\end{equation}
Consequently,
\begin{equation}
\label{eq:W-family}
W_{m,q}(x)
:=
W_{S_{m,q},m}(x)
=
\sum_{k=0}^{2q+1}x^{k(m+1)}
+x^vA_t(x),
\end{equation}
and
\begin{equation}
\label{eq:degree-family}
D:=\deg P_{m,q}=2qm+2q+2.
\end{equation}
\end{proposition}

\begin{proof}
Equation \eqref{eq:apery-set-family} is the Ap\'ery-set description in
\cite[Example~3]{CiolanGarciaSanchezMoree}.  Summing its monomials gives
\eqref{eq:W-family}.
The two subsets in \eqref{eq:apery-set-family} have respective maxima
\((2q+1)(m+1)\) and \(qm+m-1\), and
\[
(2q+1)(m+1)-(qm+m-1)
=qm+2q+2>0.
\]
Thus \(\max\Apery(S_{m,q};m)=(2q+1)(m+1)\), and
\eqref{eq:degree-from-apery} gives
\[
\deg P_{m,q}
=(2q+1)(m+1)-m+1
=2qm+2q+2.
\]
\end{proof}

The coefficient calculation in the next section will be converted into
information about the roots of \(P_{m,q}\) by the following identity.
For a formal power series \(G(x)\), let \([x^r]G(x)\) denote the
coefficient of \(x^r\).  If \(G(0)=1\), then \(\log G(x)\) denotes its
formal logarithm.

\begin{lemma}
\label{lem:power-sum}
Let \(f\in\mathbb C[x]\) have degree \(D\) and satisfy \(f(0)=1\).  If
\(\rho_1,\ldots,\rho_D\) are its roots, counted with multiplicity, then
for every integer \(r\geq1\),
\begin{equation}
\label{eq:power-sum}
-r[x^r]\log f(x)
=
\sum_{j=1}^{D}\rho_j^{-r}.
\end{equation}
Consequently, if
\[
r\left|[x^r]\log f(x)\right|>D,
\]
then \(f\) has a zero off the unit circle.
\end{lemma}

\begin{proof}
Since \(f(0)=1\), none of its roots is zero, and
\[
f(x)
=
\prod_{j=1}^{D}\left(1-\rho_j^{-1}x\right).
\]
Therefore, in \(\mathbb C[[x]]\),
\[
\log f(x)
=
-\sum_{r\geq1}\frac{x^r}{r}
\sum_{j=1}^{D}\rho_j^{-r},
\]
which proves \eqref{eq:power-sum}.  If all roots lie on the unit circle,
then
\[
r\left|[x^r]\log f(x)\right|
=
\left|\sum_{j=1}^{D}\rho_j^{-r}\right|
\leq
\sum_{j=1}^{D}|\rho_j|^{-r}
=D.
\]
The contrapositive gives the final assertion.
\end{proof}

\section{Coefficient computation and classification}

\begin{proposition}
\label{prop:decisive-coefficient}
Let \(q\geq1\) and \(m\geq2q+4\), and retain the notation of
Proposition~\ref{prop:apery-family}.  Set
\begin{equation}
\label{eq:L}
L=2(q+1)(m+1)-1.
\end{equation}
Then
\begin{equation}
\label{eq:exact-coefficient}
[x^L]\log P_{m,q}(x)
=
t-\frac1L
+
\begin{cases}
\displaystyle \frac13\binom{t-2}{2},
& q=1\text{ and }t\geq4,\\[6pt]
0,
& \text{otherwise}.
\end{cases}
\end{equation}
In particular,
\begin{equation}
\label{eq:coefficient-bound}
[x^L]\log P_{m,q}(x)
\geq
t-\frac1L>0.
\end{equation}
If \(\rho_1,\ldots,\rho_D\) are the roots of \(P_{m,q}\), counted with
multiplicity, then
\begin{equation}
\label{eq:power-sum-obstruction}
\left|\sum_{j=1}^{D}\rho_j^{-L}\right|
\geq
Lt-1
>D.
\end{equation}
\end{proposition}

\begin{proof}
Because \(L+1=(2q+2)(m+1)\), equation \eqref{eq:W-family} gives
\[
W_{m,q}(x)
=
\frac{1-x^{L+1}}{1-x^{m+1}}
+x^vA_t(x).
\]
Set
\[
Z(x)=x^v(1-x^{m+1})A_t(x).
\]
Working modulo \(x^{L+1}\) in \(\mathbb Q[[x]]\), we obtain
\begin{equation}
\label{eq:W-truncation}
W_{m,q}(x)
\equiv
\frac{1+Z(x)}{1-x^{m+1}}
\pmod{x^{L+1}}.
\end{equation}
We use the following elementary fact about formal logarithms: if
\(F,G\in1+x\mathbb Q[[x]]\) and
\(F\equiv G\pmod{x^{L+1}}\), then
\(\log F\equiv\log G\pmod{x^{L+1}}\).  Indeed,
\(F/G\equiv1\pmod{x^{L+1}}\), and hence
\[
\log F-\log G
=
\log(F/G)
\equiv0
\pmod{x^{L+1}}.
\]
Using
\[
P_{m,q}(x)
=
\frac{(1-x)W_{m,q}(x)}{1-x^m},
\]
equation \eqref{eq:W-truncation} therefore yields
\begin{equation}
\label{eq:log-P-congruence}
\begin{aligned}
\log P_{m,q}(x)
\equiv{}&
\log(1-x)+\log(1+Z(x))\\
&-\log(1-x^{m+1})-\log(1-x^m)
\pmod{x^{L+1}}.
\end{aligned}
\end{equation}
Now
\[
L\equiv-1\pmod{m+1},
\qquad
L\equiv2q+1\pmod m,
\qquad
0<2q+1<m.
\]
Thus the last two logarithms in \eqref{eq:log-P-congruence} have zero
\(x^L\)-coefficient, while
\([x^L]\log(1-x)=-1/L\).  Consequently,
\begin{equation}
\label{eq:log-P-reduction}
[x^L]\log P_{m,q}(x)
=
-\frac1L+[x^L]\log(1+Z(x)).
\end{equation}

We next compute the remaining coefficient.  The degree and the
\(x\)-adic order of \(Z\) are
\[
\deg Z
=v+(m+1)+t
=(q+2)m
<L,
\qquad
\operatorname{ord}_x Z=v,
\]
since \(L-(q+2)m=qm+2q+1>0\).  Moreover,
\[
4v-L
=2(q-1)m+6q+7
>0.
\]
Hence the linear term and all powers \(Z^j\) with \(j\geq4\) in
\[
\log(1+Z)
=Z-\frac{Z^2}{2}+\frac{Z^3}{3}-\cdots
\]
have zero \(x^L\)-coefficient.  Therefore
\begin{equation}
\label{eq:quadratic-cubic}
[x^L]\log(1+Z)
=
-\frac12[x^L]Z^2
+\frac13[x^L]Z^3.
\end{equation}

For the quadratic term, \(L-2v=m+t\), so
\begin{align}
[x^L]Z^2
&=
[x^{m+t}](1-x^{m+1})^2A_t(x)^2 \notag\\
&=
[x^{m+t}]A_t(x)^2
-2[x^{t-1}]A_t(x)^2
+[x^{t-m-2}]A_t(x)^2.
\label{eq:quadratic-shifts}
\end{align}
The first coefficient in \eqref{eq:quadratic-shifts} is zero because
\(\deg A_t^2=2t<m+t\), and the last is zero because
\(t-m-2<0\).  The coefficient of \(x^{t-1}\) in
\(A_t(x)^2\) equals \(t\): it counts the ordered pairs
\((a,b)\in\{0,\ldots,t\}^2\) satisfying \(a+b=t-1\).  Hence
\begin{equation}
\label{eq:quadratic-value}
[x^L]Z^2=-2t.
\end{equation}

It remains to compute the cubic term.  If \(q\geq2\), then
\[
3v-L=(q-2)m+4q+5>0,
\]
so \([x^L]Z^3=0\).  Suppose instead that \(q=1\).  Then
\[
L-3v=m-9=t-4.
\]
Thus \([x^L]Z^3=0\) when \(t\leq3\).  If \(t\geq4\), then
\(t-4<m+1\), and therefore every term involving a positive power of
\(x^{m+1}\) is too highly shifted.  Hence
\[
[x^L]Z^3
=
[x^{t-4}](1-x^{m+1})^3A_t(x)^3
=
[x^{t-4}]A_t(x)^3.
\]
Because \(t-4\leq t\), the upper bounds in \(A_t(x)^3\) do not become
active at this degree.  The coefficient counts the nonnegative triples
\((a,b,c)\) satisfying \(a+b+c=t-4\), and thus
\begin{equation}
\label{eq:cubic-value}
[x^L]Z^3
=
\binom{t-2}{2}.
\end{equation}
Combining \eqref{eq:log-P-reduction},
\eqref{eq:quadratic-cubic}, \eqref{eq:quadratic-value}, and
\eqref{eq:cubic-value} proves \eqref{eq:exact-coefficient}, including
the cases in which the cubic coefficient is zero.

Since \(t\geq1\), equation \eqref{eq:coefficient-bound} is positive.
Lemma~\ref{lem:power-sum} therefore gives
\[
\left|\sum_{j=1}^{D}\rho_j^{-L}\right|
=
L\left|[x^L]\log P_{m,q}(x)\right|
\geq
Lt-1
\geq
L-1.
\]
Finally,
\[
L-1
=2(q+1)(m+1)-2
=2qm+2q+2m
=D+2m-2
>D.
\]
This proves \eqref{eq:power-sum-obstruction}.
\end{proof}

\begin{proof}[Proof of Theorem~\ref{thm:classification}]
If \(m=2q+3\), then \eqref{eq:embedding-dimension} gives
\(e(S_{m,q})=3\).  The semigroup is symmetric by
\cite[Lemma~4.22]{RosalesGarciaSanchez}, and every symmetric numerical
semigroup of embedding dimension at most \(3\) is cyclotomic by
\cite[Lemma~7]{CiolanGarciaSanchezMoree}.  Hence \(S_{m,q}\) is
cyclotomic.

If \(m\geq2q+4\), Proposition~\ref{prop:decisive-coefficient} and
Lemma~\ref{lem:power-sum} show that \(P_{m,q}\) has a zero off the unit
circle.  It is therefore not a product of cyclotomic polynomials, so
\(S_{m,q}\) is noncyclotomic.
\end{proof}

\section*{Declaration of AI use}
OpenAI's ChatGPT was used during the development of this work for assistance
with literature searches, exploration and verification of mathematical arguments,
and manuscript preparation.  All mathematical claims, proofs, citations, and the
final text were independently checked and approved by the author, who takes full
responsibility for the work.


\begin{thebibliography}{99}

\bibitem{BorziHerreraPoyatosMoree}
A. Borz\`i, A. Herrera-Poyatos, and P. Moree,
\emph{Cyclotomic numerical semigroup polynomials with at most two
irreducible factors},
Semigroup Forum \textbf{103} (2021), no.~3, 812--828,
\href{https://doi.org/10.1007/s00233-021-10197-8}
{doi:10.1007/s00233-021-10197-8}.

\bibitem{CiolanGarciaSanchezMoree}
E.-A. Ciolan, P. A. Garc\'ia-S\'anchez, and P. Moree,
\emph{Cyclotomic numerical semigroups},
SIAM J.\ Discrete Math. \textbf{30} (2016), no.~2, 650--668,
\href{https://doi.org/10.1137/140989479}{doi:10.1137/140989479}.

\bibitem{CiolanGarciaSanchezHerreraPoyatosMoree}
A. Ciolan, P. A. Garc\'ia-S\'anchez, A. Herrera-Poyatos, and P. Moree,
\emph{Cyclotomic exponent sequences of numerical semigroups},
Discrete Math. \textbf{345} (2022), no.~6, Paper No.~112820, 22 pp.,
\href{https://doi.org/10.1016/j.disc.2022.112820}
{doi:10.1016/j.disc.2022.112820}.

\bibitem{HerreraPoyatosMoree}
A. Herrera-Poyatos and P. Moree,
\emph{Coefficients and higher order derivatives of cyclotomic
polynomials: old and new},
Expo. Math. \textbf{39} (2021), no.~3, 309--343,
\href{https://doi.org/10.1016/j.exmath.2019.07.003}
{doi:10.1016/j.exmath.2019.07.003}.

\bibitem{MoscarielloSammartano}
A. Moscariello and A. Sammartano,
\emph{Open problems on relations of numerical semigroups},
in M. Bre\v{s}ar, A. Geroldinger, B. Olberding, and D. Smertnig (eds.),
\emph{Recent Progress in Ring and Factorization Theory},
Springer Proc. Math. Stat., vol.~477, Springer, Cham, 2025, 365--380,
\href{https://doi.org/10.1007/978-3-031-75326-8_16}
{doi:10.1007/978-3-031-75326-8\_16}.

\bibitem{Rosales}
J. C. Rosales,
\emph{Symmetric numerical semigroups with arbitrary multiplicity and
embedding dimension},
Proc. Amer. Math. Soc. \textbf{129} (2001), no.~8, 2197--2203,
\href{https://doi.org/10.1090/S0002-9939-01-05819-1}
{doi:10.1090/S0002-9939-01-05819-1}.

\bibitem{RosalesGarciaSanchez}
J. C. Rosales and P. A. Garc\'ia-S\'anchez,
\emph{Numerical Semigroups},
Developments in Mathematics, vol.~20,
Springer, New York, 2009,
\href{https://doi.org/10.1007/978-1-4419-0160-6}
{doi:10.1007/978-1-4419-0160-6}.

\bibitem{SawhneyStoner}
M. Sawhney and D. Stoner,
\emph{On symmetric but not cyclotomic numerical semigroups},
SIAM J.\ Discrete Math. \textbf{32} (2018), no.~2, 1296--1304,
\href{https://doi.org/10.1137/17M1138479}{doi:10.1137/17M1138479}.

\end{thebibliography}
\end{document}